\documentclass[11pt]{article}
\usepackage{amsmath}
\usepackage{amssymb}
\usepackage{amsthm}
\usepackage{amsfonts}
\usepackage{graphicx}
\usepackage{pdfpages}
\usepackage{xcolor,colortbl}
\usepackage{multicol}
\usepackage{url}
\usepackage{subfigure}
\usepackage{mathtools}

\DeclarePairedDelimiter\floor{\lfloor}{\rfloor}

\definecolor{Gray}{gray}{0.92}
\definecolor{LightCyan}{rgb}{0.88,1,1}

\newcolumntype{a}{>{\columncolor{Gray}}c}
\newcolumntype{b}{>{\columncolor{white}}c}

\usepackage{bbm,epsfig,graphics,epic,color,rotating,color}

    \usepackage[left=2.5cm,right=2.5cm,top=2.3cm,bottom=2.5cm]{geometry}

\numberwithin{equation}{section}
\def\E{{\mathbb E}}

\def\P{{\mathbb P}}
\def\N{{\mathbb N}}

\newtheorem*{theorem*}{Theorem}
\newtheorem{theorem}{Theorem}[section]

\newtheorem{corollary}[theorem]{Corollary}

\newtheorem{definition}[theorem]{Definition}

\theoremstyle{definition}

\title{Urn models with two types of strategies}
\date{}
\author{Manuel Gonz\'alez-Navarrete and Rodrigo Lambert}

\begin{document}

\maketitle %
\thispagestyle{empty} %
\baselineskip=14pt

\vspace{2pt}

\begin{abstract}
We study an urn process containing red and blue balls and two different strategies to reinforce the urn. Namely, a generalized P\'olya-type strategy versus an i.i.d. one. {At each step, one of the two reinforcement strategies is chosen by flipping a coin.} We study the {asymptotic} behaviour of this urn model, {and prove} a law of large numbers, a central limit theorem and a functional limit theorem for the proportion of balls into the urn. A phase transition is also stated.
\end{abstract}

\medskip
%


\section{Introduction}
\label{sec:intro}

{Urn models and its applications are} very well exploited themes in probability theory. Its literature is extensive and an essential survey of this issue can be found in the book of Mahmoud \cite{Mah}. The applications of urn processes are related to algorithm analisys and data structures, dynamical models of social networks and evolutionary game theory, among others \cite{JK2,Mah,Pem}.

One of the earliest {works} in the literature was the paper by Eggenberger and P\'olya \cite{EP}. This so-called P\'olya urn is stated as follows. An urn starts with an initial quantity of $R_0$ red and $B_0$ blue balls and draws are made sequentially. After each draw, the ball is replaced and other $a$ balls of the same color are added to the urn. {The authors were then interested in study the current composition of the urn}, at each time $n \in \N$. {Let us consider the following notation: {the} urn is represented by the two-dimensional} vector $U_n=(R_n,B_n)$, where $R_n$ and $B_n$ represent the number of red and blue balls at time $n$, respectively. {The goal here is to study the stochastic process $(U_n)_{n \in \mathbb{N}}$.}

Posteriorly, Friedman \cite{Fri} generalized such urn model and then several extensions have been studied (for details, see \cite{Mah}). We highlight here the Bagchi-Pal mechanism, which will be used in this work (for more details on this model, see \cite{BP}). The model is represented, as usual, by the so-called reinforcement or {\it replacement matrix}

\begin{eqnarray}
\label{matrix}
M=
\left( \begin{array}{cc}
a & c  \\
b & d  \end{array} \right) \ ; \ \  \vec{R}=(a,b)^T \ ; \ \vec{B}=(c,d)^T.
\end{eqnarray}

In that case, $\vec{R}$-column determines the balls to add if the chosen color is red. This means that we put $a$ red and $b$ blue balls. If a blue-colored ball is obtained we proceed to the $\vec{B}$-column. It indicates that we add $c$ red and $d$ blue balls into the urn. As in \cite{BP}, we suppose that the urn is \emph{balanced}, which means that $a+b=c+d=K$. Therefore, at each step we add a fixed number $K$ of balls.

{A general quantity of interest in the related literature is the asymptotic proportion of red balls on the urn. Namely, the authors usually investigate $R_n/(R_n+B_n)$, as $n\to \infty$. There are several approaches to address this problem. For instance, the use of generating functions and convergence of moments \cite{Fre,JK,Mah}, embedding the urn into continuous-time branching processes \cite{AK,Jan,Pem}, ideas based on convergence results for martingales \cite{Go1,Go2,KS} or combinatorial analytic and algebraic approaches \cite{FGP,Pou}.}

In this work we consider the possibility of having different mechanisms of adding the balls into the urn at each step. In other words, we imagine that there are two players designed to reinforce the urn with new balls. In particular, suppose that one of them is a {\it good worker} employing a generalized Bagchi-Pal scheme. On the other hand, second worker is a {\it carefree one}. Then with probability $p$ and independently of the current composition of the urn he (she) chooses color blue { ($\vec{B}$-column in \eqref{matrix}). Otherwise he uses $\vec{R}$-column.} Moreover, the player who will play at time $n$ is chosen according
to a Bernoulli sequence with probability of success $\theta$. That is, at each step player $\mathcal{A}$ is chosen with probability $\theta$ and player $\mathcal{B}$ with probability $1-\theta$.

We study the influence of the strategies in the asymptotic behaviour of the model. Moreover we complement the discussions {started in \cite{GL}, proposing an application of the memory lapses property for urn models (see definition \ref{defi}).} {In the context of the present paper}, a memory lapse is a sequence of consecutive replacements done by Player $\mathcal{B}$. That is, by an i.i.d strategy, without taking the composition {(or history)} of the urn into account.

The problem is addressed by considering the ideas introduced by Janson \cite{Jan,Jan2}, which allow us to characterize convergence of moments of $(R_n, B_n)$, by interpreting the urn in Section \ref{sec:the model} as an urn with random reinforcement matrix. 

The main results include convergence theorems for the proportion of red balls in the urn {(namely $R_n/( R_n+B_n)$). The first one is a strong law of large numbers. In the sequel, we prove a central limit theorem, and a functional central limit theorem. As a corollary, we show convergence for some particular processes as an application of our results.}

The rest of this paper is organized as follows. Section \ref{sec:the model} states the model, some examples and presents the main results. The proofs for the results are provided in Section \ref{sec:proof}.

\bigskip

\section{The urn process and main results}
\label{sec:the model}

Imagine a discrete-time urn process with initial $R_0$ red and $B_0$ blue balls. The composition of the urn at time $n \in \N$ is given by $(R_n,B_n)$, where $R_n$ means the number of red and $B_n$ the number of blue balls. The quantity of balls into the urn is $T_n=Kn+T_0$, where $K=a+b$ from matrix \eqref{matrix} and $T_0=R_0+B_0$.

Assume there exist two players which follow different strategies. One of them looks to the urn composition, and the other does not. Nonetheless, both players employ the replacement matrix \eqref{matrix}, each of them following a distinct rule. Now, we describe the dynamics of this process.

Player $\mathcal{A}$ follows a (generalized) P\'olya-type regime. That is, she (he) draws a ball uniformly at random, observe its color and put it back to the urn. Then she chooses with probability $p$ the same color and with probability $1-p$ the opposite color. Finally, she uses the replacement matrix \eqref{matrix}.

Player $\mathcal{B}$ behaves simpler. He (she) chooses a color, say blue or red, with probabilities $p$ and $1-p$, respectively. Then, he puts the balls into the urn following \eqref{matrix}. We remark that player $\mathcal{B}$ behaves independently of the current urn composition.

The player choice is made as follows. At each step, a coin is flipped, and the result says who will add the next balls into the urn. Particularly, player $\mathcal{A}$ is chosen with probability $\theta$.

In this sense, this process can be interpreted as an urn process with interferences that can switch off the access for its history. In view of stochastic modelling, this phenomena could be related to an environmental factor that removes the dependence on the current composition of the system.

At this point, we introduce an independent Bernoulli sequence $(Y_n)_{n\ge 0}$ with parameters $\theta$, such that a success denotes that the player $\mathcal{A}$ was chosen, and {$Y_n=0$ indicates that we choose player $\mathcal{B}$}. At this time, we present the notion of memory lapse. {Roughly speaking, a memory lapse is a period in which} the decisions are done without taking into account the history of the process. Its definition is given as follows.

\begin{definition}
\label{defi}
A {\it memory lapse} in the urn model $(U_n)_{n\ge 0}$ is an interval $I\subset \N$ such that $Y_i=0$ for all $i\in I$ and there is no interval $J \supset I$ such that $Y_i=0$ for all $i \in J$. The length of the lapse is given by  $\ l=|I|$.
\end{definition}

In some sense, we think this period as a lapse because after these, the model always will recover the dependence on the whole past. This includes the consequences of the decisions taken in the memory lapse period. In view of the problem of this paper, it means that the law of player $\mathcal{A}$ is influenced by player's $\mathcal{B}$ choices (provided that player $\mathcal{B}$ assume the game for a period).

\subsection{Examples}

In what follows we remark some particular cases in the literature, which can be obtained by taking specific values on the parameters $p$ and $\theta$. We refer the reader to Figure \ref{fig:Region} for an illustration of these cases.

\begin{itemize}
\item[($i$)] Letting $\theta=1$, the player $\mathcal{A}$ is always chosen. Therefore \vspace{-.2cm}
\begin{itemize}
\item[($i.1$)] If $p=1$ we have the original Bagchi-Pal urn; \vspace{.2cm}
\item[($i.2$)] The case $p=0$ is like a color blind situation, where we choose the opposite color, and the replacement matrix \eqref{matrix} has its columns changed to obtain again the Bagchi-Pal mechanism.
\end{itemize}
\item[($ii$)] If $\theta=0$, player $\mathcal{B}$ is always chosen and we obtain a family of random walks with independent steps. In particular, both cases {$p=0$ or $p=1$} are deterministic. 
\item[($iii$)] For $p=1/2$ we have a class of ``symmetric'' random walks, for all $\theta \in [0,1]$. In this case, it doesn't matter what strategy is employed, always with probability $1/2$ we choose $(a,b)$ balls to reinforce the urn.
\item[($iv$)] For $p=1$ (respect. $p=0$), we obtain the Bagchi-Pal (respect. color blind Bagchi-Pal) urn with deterministic noise, depending on $\theta \in (0,1)$.
\end{itemize}

\begin{figure}[h!]
\begin{center}
\epsfig{file=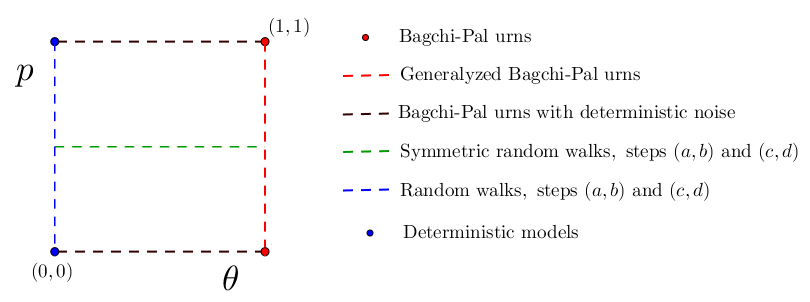, height=4.5cm}
\end{center}
\caption{Characterization of the models included in the urn process above.}
\label{fig:Region}
\end{figure}

Now, we analyse particular cases of reinforcement matrix \eqref{matrix}. We realize the evolution by plotting $(R_n,B_n)$, moving at each step in the direction of one of the column-vectors $\vec{R}$ or $\vec{B}$. For instance, let $R_0=B_0=1$ and see Figure \ref{fig:Knight} for an illustration. The first one is an urn process as an analogy with a random walk done by chess knight piece, which we call the knight random walk (KRW). Note that $\vec{R}=(2,1)^T$, $\vec{B}=(1,2)^T$ and $K=3$, as a similar model, we analyse $a=3$, $c=1$ and $K=3$.

As a consequence of the law of large numbers (see Theorem \ref{LGN}), we remark for instance, the case $p=1/2$, then the random vector $\left(\frac{R_n}{T_n}, \frac{B_n}{T_n}\right)$ converges almost surely to $\left(\frac{a+c}{2K}, 1 - \frac{a+c}{2K}\right)$, as $n$ diverges. This means that the examples  in Figure \ref{fig:Knight} fluctuate around the green dotted lines. In other words, these random walks are symmetric around the state-space's diagonal.

\begin{figure}[ht]
\begin{center}
\subfigure[]{
\epsfig{file=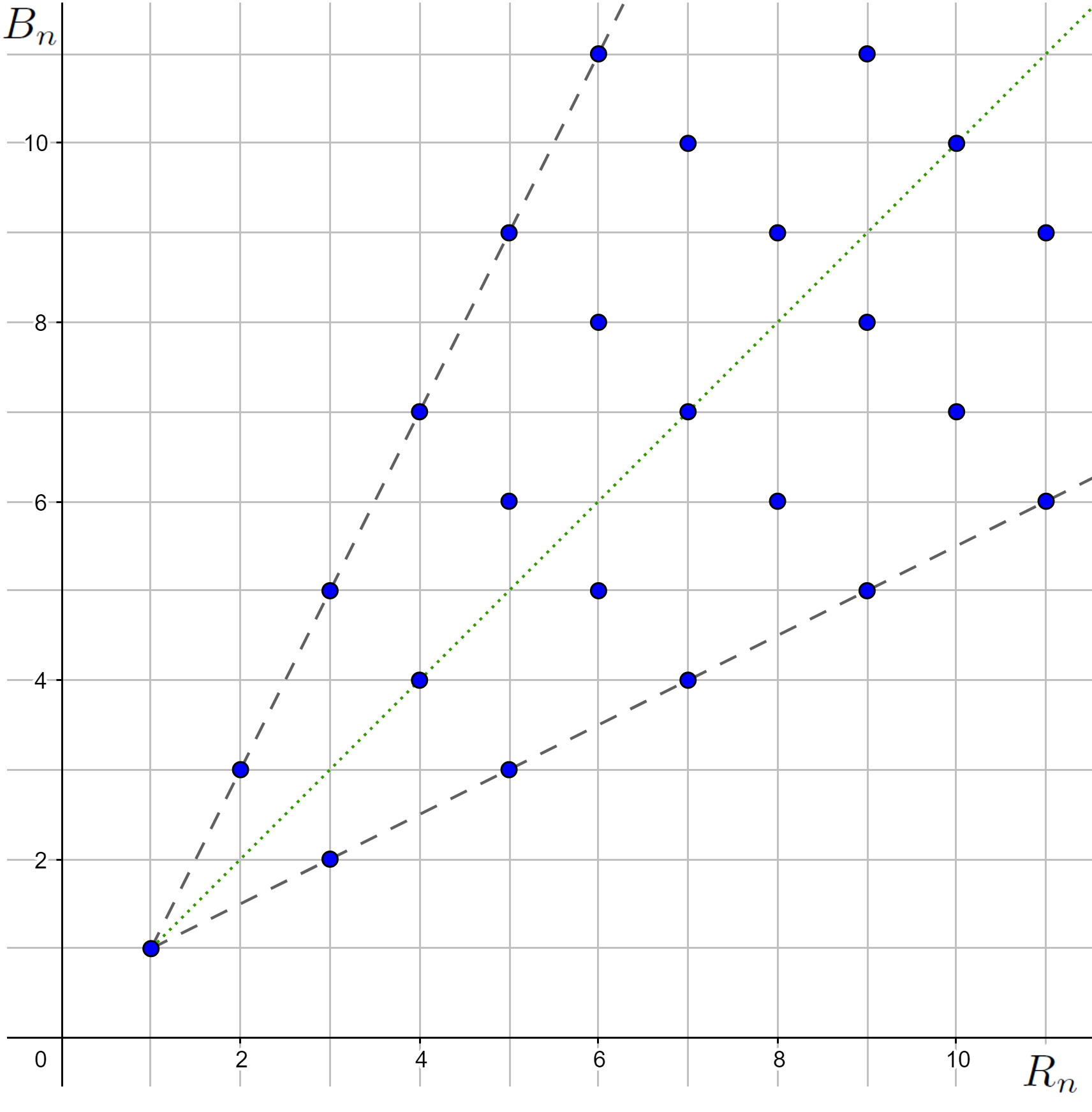, height=5cm}
}
\hspace{1cm}
\subfigure[]{\epsfig{file=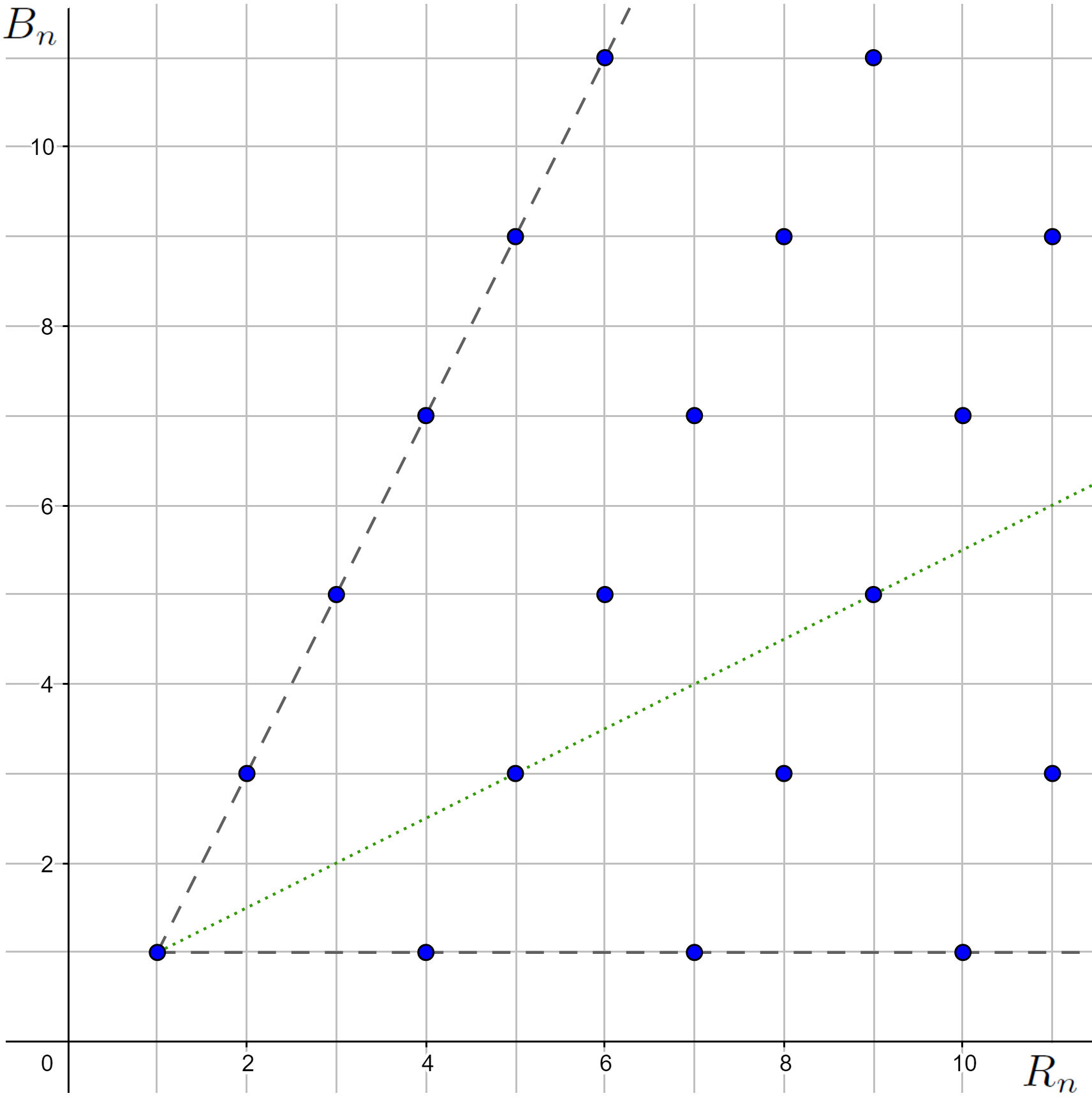, height=5cm}
}
\end{center}
\caption{State-space of the two examples with $K=3$, starting at $(R_0,B_0)=(1,1)$. The gray dashed lines represent the deterministic cases, that is, $\theta=0$, and $p\in \{0,1\}$.}
\label{fig:Knight}
\end{figure}

\subsection{Main results}

As usual, we restrict the attention to a class of well-behaved urn processes. In this line, we assume that the urn is {\it tenable}, i.e., that it is impossible to get stuck. Formally, assume that

\begin{itemize}

\item [(a)] $T_0>0$.

\item [(b)] $T_n=Kn+T_0$, for some $K\ge 1$.

\item [(c)] The urn $R_n$ is not deterministic, that is: $a\neq c$.

\item [(d)] All entries in \eqref{matrix} are {non-negative}.
\end{itemize}

{At this point we recall some terminology from the literature, and also made some remarks. Assumption (b) says that the urn is balanced, and assumption (d) is stated in order to simplify the analysis. However the case of (possible) negative entries can be studied by using similar tools.}
{

In what follows we deduce the conditional probabilities for this urn model. For simplicity, we look to the distribution of the one-dimensional stochastic process $(R_n)_{n \in \mathbb{N}}$, which count the numbers of red balls at time $n$. As stated in the urn dynamics in Section \ref{sec:the model}, we recall that $\P(Y_n=1)= \theta$, for all $n$ and that $R_n$ increases $a$-unities if and only if $\vec{R}$-column is chosen. Thus, if $Y_n=1$ player $\mathcal{A}$ is chosen. Since with probability $p$ this player uses the ``correct'' color selected from the urn (and with probability $1-p$ she(he) uses the opposite), we get
\begin{eqnarray}
\label{condprobA}
 \P(R_{n+1} = r+a | R_n = r, Y_n = 1) & =& p \frac{r}{T_n} + (1-p) \left(1-\frac{r}{T_n}\right)  \\
 & = & 1-p + (2p-1)  \frac{r}{T_n} \ .
\end{eqnarray}
Otherwise, if player $\mathcal{B}$ is chosen, then he(she) chooses $\vec{R}$ replacement vector with probability $1-p$. Therefore we get
\begin{equation}
\label{condprobB}
 \P(R_{n+1} = r+a | R_n = r, Y_n = 0) =  1-p \ .
\end{equation}

Therefore we combine the last two equations to conclude that $(R_n)_n$ evolves with transition probabilities given by
\begin{equation}
\label{condprob}
 \P(R_{n+1} = r+a | R_n = r, Y_n = y) =  1-p + (2p-1) y \frac{r}{T_n},
\end{equation}
and $ \P(R_{n+1} = r+c | R_n = r, Y_n = y) = 1-  \P(R_{n+1} = r+a | R_n = r, Y_n = y)$.

}

Now we are ready to present the main results about convergence of the random vector $(R_n, B_n)$. The first one is a strong law of large numbers, and is stated as follows.

\begin{theorem}
\label{LGN}

Let a tenable urn process {$(U_n)_{n \in \mathbb{N}}=(R_n,B_n)_{n \in \mathbb{N}}$ defined above. 
If $\theta(2p-1)(a-c)< K$ we get the following almost-surely convergence}
\begin{equation}
\label{LLNR}
\lim_{n \to \infty}\left(\frac{R_n}{T_n},\frac{B_n}{T_n}\right) = \left(\frac{pc+(1-p)a}{K - \theta(2p-1)(a-c)},\frac{K- c-(a-c)(\theta(2p-1)+(1-p))}{K - \theta(2p-1)(a-c)}\right).
\end{equation}
\end{theorem}
\medskip

We also obtain versions of the central limit theorem for it. Firstly we present two quantities that will appear in the following result. Let $\omega_1$ and $\omega_2$ two functions on the parameters of the model given by
\begin{equation}
\label{omegas}
\omega_1 = K-2c-(a-c)(\theta(2p-1)+2(1-p))\  \text{and}   \ \omega_2=c+(a-c)(1-p).
\end{equation}

Now we can present our CLTs. A general form can be stated as follows.

\bigskip

\begin{theorem}\label{CLT} Suppose that the hypothesis of Theorem \ref{LGN} are satisfied. Assume in addition that $p\neq 1/2$.

\begin{itemize}
\item[(i)]\textbf{The diffusive case:}  If $2\theta(2p-1)(a-c)< K$ then
{\small
\begin{equation*}
\label{teoi}
\frac{1}{\sqrt{n}}\left[(R_n,B_n) - n\left(\frac{pc+(1-p)a}{K - \theta(2p-1)(a-c)},1-\frac{pc+(1-p)a}{K - \theta(2p-1)(a-c)}\right) \right] \xrightarrow{d} \mathcal{N}(0,\Sigma_1)
\end{equation*}
}
with covariance matrix $\Sigma_1$ given by
\begin{equation*}
\Sigma_1 = \dfrac{\omega_1^2\alpha + 2\omega_1\omega_2\beta+\omega_2^2K^2}{(K-2\theta(a-c)(2p-1))(K-\theta(a-c)(2p-1))^2}
\left( 
\begin{array}{cc}
1
& 
-1
\\
-1
& 
1
\end{array}
\right)
\end{equation*}
where

$$
\alpha = c^2+(a^2-c^2)\left(\frac{K(1-p)+c\theta(2p-1)}{K-\theta(a-c)(2p-1)}\right) \ \ \text{and} \ \ \beta = K\left(\frac{\alpha + ac}{a+c}\right).
$$
\item[(ii)] \textbf{The critical case:} If $2\theta(2p-1)(a-c) = K$ then
{\small
\begin{equation}
\label{teoii}
\frac{1}{\sqrt{n\log(n)}}\left[(R_n,B_n) - n\left(\frac{2(pc+(1-p)a)}{K},1-\frac{2(pc+(1-p)a)}{K}\right) \right] \xrightarrow{d} \mathcal{N}(0,\Sigma_2)
\end{equation}
}
with $\Sigma_2$ given by

\begin{equation*}
\Sigma_2 = \dfrac{\omega_{1}^2\alpha_c + 2\omega_{1}\omega_{2}\beta_c+\omega_{2}^2K^2}{(K-\theta(a-c)(2p-1))^2}
\left( 
\begin{array}{cc}
1
& 
-1
\\
-1
& 
1
\end{array}
\right)
\end{equation*}
where the critical constants are given by 
$$
 \omega_1 = \frac{K}{2}-2a+2p(a-c), \ \alpha_c = c^2+2(a+c)\left((a-c)(1-p)+\frac{c}{2}\right), \ \beta_c = K\left(\frac{\alpha_c + ac}{a+c}\right)$$
and $\omega_2$ as in \eqref{omegas}.
\end{itemize}
\end{theorem}

In what follows we present the continuous-time version for Theorem \ref{CLT}. 
We recall that this convergence holds on the function space $D[0, \infty )$ of right-continuous with left-hands limits, that is, the so-called \emph{Skorohod space}.

\begin{theorem} \label{continuous}
 Suppose that the hypothesis of Theorem \ref{CLT} are satisfied.

\begin{itemize}
\item[(i)]  If $2\theta(2p-1)(a-c)< K$ then, for $n \to \infty$, in $D[0, \infty)$

\begin{equation}
\label{teo2i}
\frac{1}{\sqrt{n}}\left[(R_{\floor*{tn}},B_{\floor*{tn}}) - tn\left(\frac{pc+(1-p)a}{K - \theta(2p-1)(a-c)},1-\frac{pc+(1-p)a}{K - \theta(2p-1)(a-c)}\right) \right] \xrightarrow{d} W_t,
\end{equation}
where $W_t$ is a continuous bivariate Gaussian process with $W_0=(0,0)$, $\E(W_t)=(0,0)$ and, for $0 < s \le t$,

\begin{equation}
\label{corri}
\E(W_s W_t^T) = s\left(\frac{t}{s}\right)^{\theta(2p-1)(a-c)}\Sigma_1,
\end{equation}
where $\Sigma_1$ is given in Theorem \ref{CLT}-(i).

\item[(ii)]  If $2\theta(2p-1)(a-c) = K$ then, for $n \to \infty$, in $D[0, \infty)$

{\small
\begin{equation}
\label{teo2ii}
\frac{1}{\sqrt{n^t\log(n)}}\left[(R_{\floor*{n^t}},B_{\floor*{n^t}}) - n^t\left(\frac{2(pc+(1-p)a)}{K},1-\frac{2(pc+(1-p)a)}{K}\right) \right] \xrightarrow{d} W_t,
\end{equation}
}
where $W_t$ as above and for $0 < s \le t$,

\begin{equation}
\label{corrii}
\E(W_s W_t^T) = s\Sigma_2,
\end{equation}
where $\Sigma_2$ is given in Theorem \ref{CLT}-(ii).
\end{itemize}
\end{theorem}

The proofs of these results will be given in Section \ref{sec:proof}. The following result is an immediate consequence of the previous theorem, and characterizes some processes obtained for particular matrices \eqref{matrix}. In the last two cases, the parametric space will be splitted by the critical curve $p_c = \frac{K}{a-c} \frac{1}{4\theta} + \frac{1}{2}$ (see Figure \ref{fig:critical} for an illustration). We state it here without proof.

\bigskip

\begin{corollary} 
\label{coro}
The following three results hold

\begin{itemize}
 \item[(i)] Let $a=d=2$ and $b=c=1$ the KRW. Suppose that the hypothesis of Theorem \ref{CLT} are satisfied. Therefore, for all $0\le p, \theta \le 1$, for $n \to \infty$, in $D[0, \infty)$

\begin{equation*}
\frac{1}{\sqrt{n}}\left[(R_{\floor*{tn}},B_{\floor*{tn}}) - tn\left(\frac{2-p}{3 - \theta(2p-1)},1-\frac{2-p}{3 - \theta(2p-1)}\right) \right] \xrightarrow{d} W_t,
\end{equation*}
where $W_t$ as defined in Theorem \ref{continuous} and, for $0 < s \le t$, $\E(W_s W_t^T) = s\left(\frac{t}{s}\right)^{\theta(2p-1)}\Sigma_1$, with $\Sigma_1$ as in Theorem \ref{CLT}-(i), determined by 
$$\omega_1= (1 - \theta)(2p-1), \ \omega_2=2-p, \ \alpha=1+ 3 \left(\frac{3(1-p)+\theta(2p-1)}{3-\theta(2p-1)}\right) \ \text{and} \ \beta=2+\alpha.$$

\item[(ii)] Let $a=2, c=0$ and $K=3$. Then there is a critical curve $p_c= \frac{3}{8\theta} + \frac{1}{2}$ such that for all $p < p_c$, in $D[0, \infty)$
\begin{equation*}
\frac{1}{\sqrt{n}}\left[(R_{\floor*{tn}},B_{\floor*{tn}}) - tn\left(\frac{2(1-p)}{3 - 2\theta(2p-1)},1-\frac{2(1-p)}{3 - 2\theta(2p-1)}\right) \right] \xrightarrow{d} W_t,
\end{equation*}
as $n$ diverges, where $W_t$ as defined in Theorem \ref{continuous} and, for $0 < s \le t$, $\E(W_s W_t^T) = s\left(\frac{t}{s}\right)^{2\theta(2p-1)}\Sigma_1$, with $\Sigma_1$ as in Theorem \ref{CLT}-(i), determined by 
$$\omega_1= 1+2(1 - \theta)(2p-1), \ \omega_2=2(1-p), \ \alpha=4 \left( \frac{3(1-p)}{3-2\theta(2p-1)}\right) \ \text{and} \ \beta=\frac{3}{2}\alpha.$$

Moreover, if $p=p_c$ and $n \to \infty$, in $D[0, \infty)$
{\small
\begin{equation*}
\frac{1}{\sqrt{n^t\log(n)}}\left[(R_{\floor*{n^t}},B_{\floor*{n^t}}) - n^t\left(\frac{4(1-p)}{3},1-\frac{4(1-p)}{3}\right) \right] \xrightarrow{d} W_t,
\end{equation*}
}
where $W_t$ as in Theorem \ref{continuous} and for $0 < s \le t$, $\E(W_s W_t^T) = s \Sigma_2$ with $\Sigma_2$ as in Theorem \ref{CLT}-(ii), determined by
$$\omega_1= \frac{3}{2}+4(p-1), \ \omega_2=2(1-p), \ \alpha_c=8(1-p) \ \text{and} \ \beta_c=\frac{3}{2}\alpha_c.$$
\item[(iii)] For any model with $K= a-c \ge 1$, there is a critical curve at $p_c= \frac{1}{4\theta} + \frac{1}{2}$. Then, for all $p < p_c$ as $n \to \infty$, in $D[0, \infty)$

\begin{equation*}
\frac{1}{\sqrt{n}}\left[(R_{\floor*{tn}},B_{\floor*{tn}}) - tn\left(\frac{1-p}{1 - \theta(2p-1)},1-\frac{1-p}{1- \theta(2p-1)}\right) \right] \xrightarrow{d} W_t,
\end{equation*}
with $W_t$ as defined in Theorem \ref{continuous} and, for $0 < s \le t$, $\E(W_s W_t^T) = s\left(\frac{t}{s}\right)^{K\theta(2p-1)}\Sigma_1$, with $\Sigma_1$ as in Theorem \ref{CLT}-(i), determined by 
$$\omega_1= K(1 - \theta)(2p-1), \ \omega_2=K(1-p), \ \text{and} \ \alpha=\beta=K^2 \left( \frac{(1-p)}{1-\theta(2p-1)}\right).$$

Moreover, if $p=p_c$ and $n \to \infty$, in $D[0, \infty)$
{\small
\begin{equation*}
\frac{1}{\sqrt{n^t\log(n)}}\left[(R_{\floor*{n^t}},B_{\floor*{n^t}}) - n^t\left(2(1-p),1-2(1-p)\right) \right] \xrightarrow{d} W_t,
\end{equation*}
}
with $W_t$ as in Theorem \ref{continuous} and for $0 < s \le t$, $\E(W_s W_t^T) = s \Sigma_2$ with $\Sigma_2$ as in Theorem \ref{CLT}-(ii), determined by
$$\omega_1= \frac{K}{2}(4p-3), \ \omega_2=K(1-p), \ \text{and} \ \alpha_c=\beta_c=2K^2(1-p).$$

\end{itemize}
\end{corollary}

The last Corollary is related to the elephant random walks studied in \cite{BB}. In particular, if $K=1$ and $\theta=1$, we recover their results. In general, the critical curve appears in models which $K \in [a-c, 2(a-c)]$.

\begin{figure}[h!]
\begin{center}
\epsfig{file=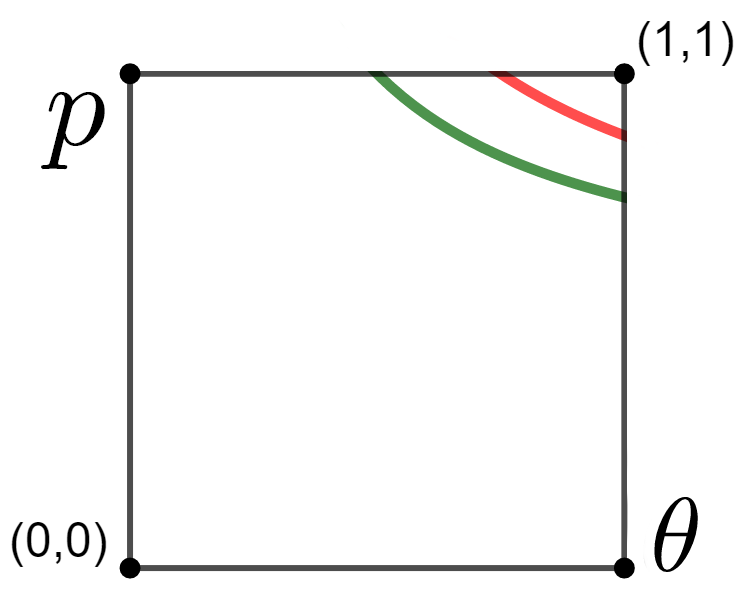, height=4cm}
\end{center}
\caption{Phase transition {diagram.} Red and green lines represent, respectively, the critical curves for items $(ii)$ and $(iii)$ of Corollary \ref{coro}.}
\label{fig:critical}
\end{figure}

\bigskip

\section{Proofs}
\label{sec:proof}

In \cite{Jan}, Janson established a powerful approach to obtain functional limit theorems for branching processes, which can be applied to the study of P\'olya-type urns. In this sense, recent works \cite{Jan2,JP} were dedicated to analyze limit behaviour of the random vector $(R_n,B_n)$ in the context of the so-called random replacement matrix. In that case, the {columns} are given by random vectors. In other words, at each step the number of balls added into the urn is given by the realization of random variables.

In this sense, let the random vectors $\xi_i = (\xi_{i1},\xi_{i2})^T$, for $i=1,2$, which represent the {columns} of the random replacement matrix. That is, if we pick a red ball, then we add to the urn a random number of balls, sampled from the vector $(\xi_{11},\xi_{12})^T$, otherwise picking a blue ball we add a sample of the random vector $(\xi_{21},\xi_{22})^T$.

{We recall that, in view of the urn models with random reinforcement matrix approach proposed by \cite{Jan, Jan2, JP} it suffices to know the moments of ${\xi_{ij}} \ ;  \ i,j\in\{1,2\}$.} Therefore, we aim to obtain the distribution of random variables $\xi_{ij}$ in such a way that the associated urn model, with random reinforcement matrix, has the same law as the urn process defined in Section \ref{sec:the model}. In this direction, we need to construct an appropriated random reinforcement matrix. For this, we assume that $\xi_{i} \in \{(a,b)^T,(c,d)^T\}$, for $i=1,2$, and let

\begin{equation}
\label{Proba}
\begin{array}{ll}
\P ( \xi_{1} = (a,b)^T) = (2p-1)\theta + (1-p) \ ; \\
\P ( \xi_{1} = (c,d)^T) = 1 -\P ( \xi_{1} = (a,b)^T) \ ;\\
\P ( \xi_{2} = (a,b)^T) = 1-p \ ; \\
\P ( \xi_{2} = (c,d)^T) = 1-\P ( \xi_{2} = (a,b)^T)  \ .
\end{array}
\end{equation}

By a straightforward calculation we conclude that the conditional probabilities for $(R_n, B_n)$ of the urn model in Section \ref{sec:the model} are equal to those obtained by the urn process with random reinforcement matrix defined by random variables $\{\xi_{ij}\}$ as given by \eqref{Proba}.

Then, we define the matrix
\begin{equation}
\label{Jan}
 A = \left(
\begin{array}{cc}
\E(\xi_{11}) & \E(\xi_{21}) \\
\E(\xi_{12}) & \E(\xi_{22})
\end{array}
\right) \ .
\end{equation}

From \eqref{Proba} we get 

\begin{equation}
\label{Exp}
\begin{array}{c}
\E(\xi_{11}) = c + (a-c) (\theta (2p-1) + (1-p)),  \  \E(\xi_{12}) = K - \E(\xi_{11}),\\[0.4cm]
\E(\xi_{21}) = c + (a-c)(1-p) \ \text{ and }  \  \E(\xi_{22}) = K-\E(\xi_{21}).
\end{array}
\end{equation}

Let be $\lambda_1$ and $\lambda_2$ the eigenvalues of matrix $A$ in \eqref{Jan}. A general form is given by
\begin{equation}
\lambda_1 = \dfrac{\E(\xi_{11}) + \E(\xi_{22}) + \sqrt{\E(\xi_{11})^2 + 4 \E(\xi_{12}) \E(\xi_{21}) - 2 \E(\xi_{11}) \E(\xi_{22}) + \E(\xi_{22})^2}}{2}
\end{equation}
and
\begin{equation}
\lambda_2 = \dfrac{\E(\xi_{11}) + \E(\xi_{22}) - \sqrt{\E(\xi_{11})^2 + 4 \E(\xi_{12}) \E(\xi_{21}) - 2 \E(\xi_{11}) \E(\xi_{22}) + \E(\xi_{22})^2}}{2} \ .
\end{equation}

Finally, by a straightforward calculation we conclude that
\begin{equation}
\label{autova}
\lambda_1 = K \ \ \text{ and } \ \ \lambda_2 =  \theta (a-c)(2p-1).
\end{equation}
Note that, $\lambda_2$ is actually the difference $\E(\xi_{11}) - \E(\xi_{21})$, and $\lambda_1$ is given by $\E(\xi_{i1}) + \E(\xi_{i2}), i=1,2$.

Using the normalization due to \cite{Jan,Jan2}, the correspondent right and left eigenvectors are given by
\begin{equation}
\label{autove1}
v_1 = \frac{1}{K-\theta (a-c)(2p-1)} {{c+(a-c)(1-p)}\choose {K-c-(a-c)(\theta(2p-1)+1-p)}}  
\end{equation}
\begin{equation}
\label{autove2}
v_2 =   \frac{1}{K-\theta (a-c)(2p-1)} {{1}\choose {-1}}
\end{equation}
and 
\begin{equation}
\label{autove3}
u_1 =    {{1}\choose {1}} \ \ ; \ \ 
u_2 =   {{K - c - (a-c) (\theta (2p-1) + (1-p))}\choose {c + (a-c)(1-p)}}.
\end{equation}

The last equations provide the spectrum of $A$. By standard arguments, it follows that conditions (A1)-(A6) and also the non-extinction condition from \cite{Jan} are satisfied. Since these are the main ingredients to obtain the limit theorems in \cite{Jan,Jan2}, we are then ready to present the explicit proofs.

\begin{proof}[Proof of Theorem \ref{LGN}] Theorem \textbf{3.21} from \cite{Jan} states that 
$$
n^{-1}(R_n, B_n) \longrightarrow \lambda_1 v_1 \ .
$$

Since $T_n = nK + T_0$, we use \eqref{autova} and \eqref{autove1} in the last display to finish the proof.
\end{proof}

\begin{proof}[Proof of Theorem \ref{CLT}]

For item $(i)$, we remark that Theorem \textbf{3.22} of \cite{Jan} to our case says that the limiting covariance matrix is given by
\begin{equation}\label{integral}
\Sigma_1 = \int_{0}^{\infty}\psi_A(s) B \psi_A(s)^T e^{- s}ds -  v_1 v_1^T \ ,
\end{equation}
where $\psi_A(s) = e^{sA} - v_1 (1 \ 1)\int_{0}^{s}e^{tA} dt$  and $B=\sum_{i=1}^2 v_{1i}B_i$ , with $B_i = \E[\xi_i \xi_i^T]$. To compute the $B$-matrix, consider \eqref{Proba}, using notation $\P ( \xi_{i} = (a,b)^T)= \P ( \xi_{i1} = a)$ and $\P ( \xi_{i} = (c,d)^T)= \P ( \xi_{i1} = c)$, we obtain
\begin{equation}
\E[\xi_i \xi_i^T] = \left(\begin{array}{c}a\\b\end{array}\right) \left(\begin{array}{c}a\\b\end{array}\right)^T\P(\xi_{i1} = a) + \left( \begin{array}{c}c\\d\end{array}\right)\left( \begin{array}{c}c\\d\end{array}\right)^T \P(\xi_{i1} = c),
\end{equation}
for $i =1,2$. Now use the fact that $\P ( \xi_{i1} = c) = 1 - \P ( \xi_{i1} = a)$, then
\begin{equation*}
B =\left( \begin{array}{c}c\\d\end{array}\right)\left( \begin{array}{c}c\\d\end{array}\right)^T + \left[ \left(\begin{array}{c}a\\b\end{array}\right) \left(\begin{array}{c}a\\b\end{array}\right)^T- \left( \begin{array}{c}c\\d\end{array}\right)\left( \begin{array}{c}c\\d\end{array}\right)^T \right]\left( v_{11}\P(\xi_{11} = a) + v_{12}\P(\xi_{21} = c)\right).
\end{equation*}
Note that, by $b=K-a$ and $d=K-c$, we obtain
\begin{equation*}
\left( \begin{array}{c}c\\d\end{array}\right)\left( \begin{array}{c}c\\d\end{array}\right)^T = c \left( \begin{array}{cc} 1 & -1 \\ -1 & 1 \end{array} \right) + Kc \left( \begin{array}{cc}  0 & 1 \\ 1 & -2 \end{array} \right) + K^2 \left( \begin{array}{cc}  0 & 0 \\ 0 & 1 \end{array} \right),
\end{equation*}
and
\begin{equation*}
\left(\begin{array}{c}a\\b\end{array}\right) \left(\begin{array}{c}a\\b\end{array}\right)^T- \left( \begin{array}{c}c\\d\end{array}\right)\left( \begin{array}{c}c\\d\end{array}\right)^T =(a^2-c^2) \left( \begin{array}{cc} 1 & -1 \\ -1 & 1 \end{array} \right) + K(a-c) \left( \begin{array}{cc}  0 & 1 \\ 1 & -2 \end{array} \right).
\end{equation*}
Finally
\begin{equation}
B = \alpha \left( \begin{array}{cc} 1 & -1 \\ -1 & 1 \end{array} \right) + \beta \left( \begin{array}{cc}  0 & 1 \\ 1 & -2 \end{array} \right) + K^2 \left( \begin{array}{cc}  0 & 0 \\ 0 & 1 \end{array} \right),
\end{equation}
where $\alpha$ and $\beta$ as in Theorem \ref{teoi}(i). Therefore we are able to compute $\Sigma_1$. We compute directly $\psi_A(s)$ by using $A= V \Lambda V^{-1}$, where
$$V=\left( \begin{array}{cc} u_{11} & v_{11} \\ u_{12} & v_{12} \end{array} \right) \ \text{ and } \ \Lambda = \left( \begin{array}{cc}  \lambda_1 & 0 \\ 0 & \lambda_2 \end{array} \right), $$
then $e^A = V e^{\Lambda}V^{-1}$, and finally we obtain

%

\begin{eqnarray*}
\Sigma_1 & = & \frac{u_2^TBu_2}{\lambda_1-2\lambda_2} v_2v_2^T  \\[0.3cm]
         & = & \frac{\alpha u_2^T\left( \begin{array}{cc} 1 & -1 \\ -1 & 1 \end{array} \right) u_2+ \beta u_2^T \left( \begin{array}{cc}  0 & 1 \\ 1 & -2 \end{array} \right)u_2 + K^2 u_2^T \left( \begin{array}{cc}  0 & 0 \\ 0 & 1 \end{array} \right)u_2}{(K-2\theta(a-c)(2p-1))(K-\theta(a-c)(2p-1))^2}\left(\begin{array}{cc} 1 & -1 \\ -1 & 1 \end{array}\right) \\[0.3cm]
                  & = & \frac{(u_{21}-u_{22})\left((u_{21}-u_{22})\alpha + 2u_{22}\beta\right)+u_{22}^2K^2}{(K-2\theta(a-c)(2p-1))(K-\theta(a-c)(2p-1))^2}\left(\begin{array}{cc} 1 & -1 \\ -1 & 1 \end{array}\right) \ , 
\end{eqnarray*}
where $u_{21}-u_{22}=\omega_1$ and $u_{22}=\omega_2$ as in \eqref{omegas}, it finishes the proof of item $(i)$. Before proving item $(ii)$ let us recall a special case of covariance matrix $\Sigma_1$. Let $K=a-c \geq 1$ and get that $b=d=0$. In this situation, $\alpha=\beta$ and $u_2 = K \left(p-\theta(2p-1),1-p\right)^T$ as in Corollary \ref{coro}(iii), therefore 
$$
\Sigma_1 = \frac{K(1-p)\left[(2p-1)(1-\theta)+(1-p)\right]}{(1-2\theta(2p-1))(1-\theta(2p-1))^2} \left(\begin{array}{cc} 1 & -1 \\ -1 & 1 \end{array}\right) \ .
$$

Now let's prove item $(ii)$. For this we use the fact that $A$ is diagonalizable combined with Theorem \textbf{3.23} and Corollary \textbf{5.3-(i)} of \cite{Jan} to obtain the following limiting covariance matrix

\begin{equation}\label{critical}
\Sigma_2 = u_2^TBu_2 v_2 v_2^T \ .
\end{equation}
By a direct computation of the right-hand term in last equation, we conclude the proof.
\end{proof}

\begin{proof}[Proof of Theorem \ref{continuous}] We start by proving $(i)$. The convergence statement is directly provided by Theorem $\textbf{3.31}-(i)$ of \cite{Jan}. If we combine this with Remark \textbf{5.7} of the same paper, we get that the limiting covariance matrix is given by

\begin{equation}
\E(W_s W_t^T) = s \Sigma_1 e^{\log(t/s) A^T}.
\end{equation}
Once more we use the fact that $A$ is diagonalizable to compute directly the right hand term of the above equation. This proves $(i)$. For the convergence assertion $(ii)$ we apply Theorem $\textbf{3.31}-(ii)$ of \cite{Jan}, with $d=0$, since $A$ is diagonalizable. Now let us compute its covariance matrix
$$
\E(W_s W_t^T) = s v_2u_2B u_2^Tv_2^T \,
$$
and once again, a direct computation finishes the proof.

\end{proof}

\bigskip

\subsection*{Acknowledgements}
MGN was supported by Funda\c{c}\~ao de Amparo \`a Pesquisa do Estado de S\~ao Paulo, FAPESP (grant 2015/02801-6). 
RL is partially supported by FAPESP (grant 2014/19805-1), and CNPQ PDJ grant (process 406324/2017-4). He kindly thanks Edmeia da Silva for her teachings in linear algebra. This work is part of the Universal FAPEMIG project ``Din\^amica de recorr\^encia para shifts aleat\'orios e processos com lapsos de mem\'oria" (process APQ-00987-18).

\medskip

{\scriptsize
{\sc Departamento de Estad\'i{}stica, Universidad del B\'io-B\'io. Avda. Collao 1202, CP 4051381, Concepci\'on, Chile.} E-mail address: magonzalez@ubiobio.cl\\

{\sc Faculdade de Matem\'atica, Universidade Federal de Uberl\^andia, Av. Jo\~ao Naves de Avila, 2121, CEP 38408-100, Uberl\^andia, MG, Brasil.} E-mail address: rodrigolambert@ufu.br


\smallskip

\begin{thebibliography}{10}
%
\bibitem{AK} Athreya, K. and Karlin, S. (1968)
\newblock {\it Embedding of urn schemes into continuous time Markov branching processes and related limit theorems.}
Ann. Math. Statist. \textbf{39}, 1801--1817.

\bibitem{BP} Bagchi, A. and Pal, A.K. (1985)
\newblock {\it Asymptotic normality in the generalized P\'olya-Eggenberger urn model, with an application to computer data structures.}
\newblock SIAM J. Algebraic Discrete Methods \textbf{6}, 394--405.

\bibitem{BB} Baur, E. and Bertoin, J. (2016)
\newblock {\it Elephant random walks and their connection to P\'olya-type urns.}
\newblock Phys. Rev. E \textbf{94}, 052134.

\bibitem{EP} Eggenberger, F. and P\'olya, G. (1923)
\newblock {\it \"{U}ber die Statistik verketteter Vorg\"{a}nge.}
Z. angew. Math Mech. \textbf{3}, 279--289.

%


\bibitem{FGP} Flajolet, P., Gabarr\'o, J. and Pekari, H. (2005)
\newblock {\it Analytic urns.}
\newblock Ann. Probab. \textbf{33}, 1200--1233.


\bibitem{Fre} Freedman, D.A. (1965)
\newblock {\it Bernard Friedman's urn.}
Ann. Math. Statist. \textbf{36(3)}, 956--970.

\bibitem{Fri} Friedman, B. (1949)
\newblock {\it A simple urn model.}
\newblock Comm. Pure Appl. Math. \textbf{2}, 59--70.


\bibitem{GL} Gonz\'alez-Navarrete, M. and Lambert, R. (2018)
{\em Non-Markovian random walks with memory lapses.}
J. Math. Phys. \textbf{59}, 113301.

\bibitem{GL2} Gonz\'alez-Navarrete, M. and Lambert, R. (2018)
{\em The diffusion of opposite opinions in a random-trend environment},
\newblock arXiv:1811.12070



\bibitem{Go1} Gouet, R. (1989)
\newblock {\it A martingale approach to strong convergence in a generalized P\'olya-Eggenberger urn model.}
\newblock Statis. Probab. Lett. \textbf{8}, 225--228.

\bibitem{Go2} Gouet, R. (1993)
\newblock {\it Martingale functional central limit theorems for a generalized P\'olya urn.}
\newblock Ann. Probab. \textbf{21}, 1624--1639.
%
%

\bibitem{Jan} Janson, S. (2004)
\newblock {\it Functional limit theorems for multitype branching processes and generalized P\'olya urns.}
 Stoch. Proc. Appl. \textbf{110}(2), 177--245.
 
\bibitem{Jan2} Janson, S. (2016)
\newblock {\it Mean and variance of balanced P\'olya urns.}
\newblock arXiv:1602.06203
 
\bibitem{JP} Janson, S. and Pouyanne, N (2018)
\newblock {\it Moment convergence of balanced P\'olya processes.}
Electron. J. Probab. \textbf{23}, 33.
 
 
\bibitem{JK} Johnson, N.L. and Kotz, S. (1976)
\newblock {\it Two variants of P\'olya's urn models (towards a rapprochement between combinatorics and probability theory).}
\newblock Am. Stat. \textbf{13}, 186--188.

\bibitem{JK2} Johnson, N.L. and Kotz, S. (1977)
{\em  Urn Models and Their Application.}
Wiley, New York.


\bibitem{KS} Kuba, M. and Sulzbach, H. (2017)
\newblock {\it On martingale tail sums in affine two-color urn models with multiple drawings}.
 J. Appl. Probab. \textbf{54}(1), 96--117.
 
\bibitem{Mah} Mahmoud, H. (2008)
{\em P\'olya urn models.}
CRC Press.


\bibitem{Pem} Pemantle, R. (2007)
{\em A survey of random processes with reinforcement.}
Probab. Surveys \textbf{4}, 1--79.

\bibitem{Pou} Pouyanne, N. (2008)
{\em An algebraic approach to P\'olya processes.}
Ann. Inst. Henri Poincar\'e Probab. Stat. \textbf{44}(2), 293--323.

%


\end{thebibliography}
\end{document}